\documentclass[a4paper,12pt]{article}

\usepackage{amscd}
\usepackage{amsfonts}
\usepackage{amsmath}
\usepackage{amssymb}
\usepackage{amsthm}
\usepackage[T1]{fontenc} 
\usepackage{here} 
\usepackage{mathrsfs} 
\usepackage{txfonts} 
\usepackage[all]{xy} 
\usepackage{algorithm} 
\usepackage{algpseudocode} 

\allowdisplaybreaks

\theoremstyle{plain}
\newtheorem{thm}{Theorem}[section]

\newtheorem{prp}[thm]{Proposition}

\theoremstyle{definition}
\newtheorem{dfn}[thm]{Definition}

\newtheorem{exm}[thm]{Example}

\newcommand{\vs}[1][0.2]{\vspace{#1in}\noindent\ignorespaces}
\newcommand{\ba}{\begin{array*}}
\newcommand{\ea}{\end{array*}}
\newcommand{\be}{\begin{eqnarray*}}
\newcommand{\ee}{\end{eqnarray*}}
\newcommand{\bi}{\begin{itemize}}
\newcommand{\ei}{\end{itemize}}
\newcommand{\bb}{\vs\begin{itembox}}
\newcommand{\eb}{\end{itembox}}
\newcommand{\bc}{\begin{center}}
\newcommand{\ec}{\end{center}}
\newcommand{\bs}{\vs\begin{screen}}
\newcommand{\es}{\end{screen}}

\def\ens#1{{\mathchoice{\left\{ #1 \right\}}{\{ #1 \}}{\{ #1 \}}{\{ #1 \}}}}
\def\set#1#2{{\mathchoice{\left\{ #1 \middle| #2 \right\}}{\{ #1 \mid #2 \}}{\{ #1 \mid #2 \}}{\{ #1 \mid #2 \}}}}
\def\r#1{\text{\rm #1}}
\def\t#1{\text{#1}}
\def\Bigv#1{\left| #1 \right|}
\def\v#1{{\mathchoice{\Bigv{#1}}{| #1 |}{| #1 |}{| #1 |}}}
\def\Bign#1{\left\| #1 \right\|}
\def\n#1{{\mathchoice{\Bign{#1}}{\| #1 \|}{\| #1 \|}{\| #1 \|}}}

\newcommand{\bN}{\mathbb{N}}

\newcommand{\bQ}{\mathbb{Q}}
\newcommand{\bR}{\mathbb{R}}

\newcommand{\bZ}{\mathbb{Z}}

\newcommand{\rC}{\r{C}}

\newcommand{\rM}{\r{M}}

\newcommand{\N}{\bN}
\newcommand{\Q}{\bQ}
\newcommand{\R}{\bR}
\newcommand{\Z}{\bZ}

\newcommand{\Fp}{\mathbb{F}_p}

\newcommand{\Qp}{\mathbb{Q}_p}

\newcommand{\Zp}{\mathbb{Z}_p}

\algnewcommand\algorithmicbreak{{\bf break}}
\algnewcommand\Break{\algorithmicbreak{}}
\algnewcommand\algorithmiccontinue{{\bf continue}}
\algnewcommand\Continue{\algorithmiccontinue{}}

\title{$p$-adic Character Neural Network}
\author{Tomoki Mihara}
\date{}

\begin{document}

\maketitle
\begin{abstract}
We propose a new frame work of $p$-adic neural network. Unlike the original $p$-adic neural network by S.\ Albeverio, A.\ Khrennikov, and B.\ Tirrozi using a family of characteristic functions indexed by hyperparameters of precision as activation functions, we use a single injective $p$-adic character on the topological Abelian group $\Zp$ of $p$-adic integers as an activation function. We prove the $p$-adic universal approximation theorem for this formulation of $p$-adic neural network, and reduce it to the feasibility problem of polynomial equations over the finite ring of integers modulo a power of $p$.
\end{abstract}
\tableofcontents

\section{Introduction}
\label{Introduction}

Let $p$ be a prime number. The notion of $p$-adic numbers is invented by K.\ Hensel in 1897 in \cite{Hen97}, and plays a central role in modern number theory. Recently, the $p$-adic numbers appears also in other branches of science, because of many significant similarities to and differences from the real numbers. Application of the $p$-adic numbers also appears in computer science. For example, S.\ Albeverio, A.\ Khrennikov, and B.\ Tirrozi studied $p$-adic neural network in \cite{AKT99} and \cite{KT00}, P.\ E.\ Bradley studied dendrograms and clusterings using $p$-adic numbers in \cite{Bra08} and \cite{Bra09}, and so on. Introduction of \cite{Bra25} explains the history well. Especially, $p$-adic equations and $p$-adic optimisation related to $p$-adic regression and $p$-adic neural networks are recently studied as a frontier topic (cf.\ \cite{ZZ23}, \cite{ZZB24}, \cite{BMP25}, \cite{Zub25-1}, \cite{Zub25-2}, \cite{Ngu25}, and \cite{Mih26-1}). We further formulated a $p$-adic counterpart of principal component analysis in \cite{Mih26-2}.

\vs
In this paper, we propose a new formulation of a $p$-adic neural network. In order to explain benefits of the new formulation, we briefly recall the original $p$-adic neural network in \cite{KT00}. Let $I$ a finite set, $X = (\vec{x}_i)_{i \in I} \in (\Qp^N)^I$ a sequence of sample points with $N \in \N$, $Y = (\vec{y}_i)_{i \in I} \in (\Qp^M)^I$ a sequence of observed values corresponding to $X$ with $M \in \N$, and $D \in \N_{> 0}$ a hyperparameter for dimension. For a set $S$, a map $f \colon S \to \Qp$, and an $\vec{s} = (s_n)_{n=0}^{N-1} \in S^N$, we abbreviate $(f(s_n))_{n=0}^{N-1}$ to $f(\vec{s})$.

\vs
The $p$-adic neural network in \cite{KT00} is given as the following optimisation problem under the assumptions $p = 2$, $X \in (\ens{0,1}^N)^I$, and $Y \in (\ens{0,1}^M)^I$:
\be
\begin{array}{ll}
\textrm{minimise} & \n{(\vec{y}_i - (1-1_{p^k \Zp}(A \vec{x}_i)))_{i \in I}} \\
\textrm{subject to} & A \in \rM_{D,N}(\Zp),
\end{array}
\ee
where $1_{p^k \Zp}$ is the characteristic function of the closed ball $p^k \Zp \subset \Qp$ centred at $0$ of radius $\v{p}^k$ for a fixed hyperparameter $k \in \N \setminus \ens{0}$ for precision.

\vs
The $p$-adic universal approximation theorem holds when we vary $k$. There are two natural candidates of formulations with varying $k$: employing $1 - 1_{p^k \Zp}$ for all $k \in \N \setminus \ens{0}$ as activation functions, or extending coefficients of $A$ to $\Qp$ so that the single function $1_{\Zp}$ is sufficient for an activation function. Both formulations are naturally extended to a general setting without the assumptions $p = 2$, $X \in (\ens{0,1}^N)^I$, and $Y \in (\ens{0,1}^M)^I$. For example, the resulting optimisation problem for the latter formulation is given in the following:
\be
\begin{array}{ll}
\textrm{minimise} & \n{(\vec{y}_i - C 1_{\Zp}(A \vec{x}_i + \vec{b}))_{i \in I}} \\
\textrm{subject to} & (A,\vec{b},C) \in \rM_{D,N}(\Qp) \times \Qp^D \times \rM_{M,D}(\Qp)
\end{array}
\ee
We note that $\vec{b}$ can be removed, but we consider it because of the analogy to the real neural network. Although this generalised $p$-adic neural network based on $1_{\Zp}$ theoretically works by the corresponding $p$-adic universal approximation theorem, the multiplication by $A \in \rM_{D,N}(\Qp)$ does not necessarily preserve a compact subspace including the image of the sequence $X$ of sample points, and hence prevents $p$-adic optimisation methods based on topological generators, e.g.\ orthonormal Schauder bases, of the Banach $\Qp$-algebra of $p$-adic functions on a compact topological space. In particular, uniform approximation by polynomials based on $p$-adic Stone--Weierstrass theorem does not work here.

\vs
In order to solve this issue, we introduce a new $p$-adic neural network based on a $p$-adic character on $\Zp$, i.e.\ a continuous group homomorphism $\Zp \to \Qp^{\times}$. Let $\chi$ be an injective $p$-adic character, e.g.\ the $p$-adic exponential function $\exp_p(qx)$ rescaled by $q \in p \Zp$ defined as $4$ when $p = 2$ and $p$ otherwise. {\it The $p$-adic character neural network} is the following optimisation problem:
\be
\begin{array}{ll}
\textrm{minimise} & \n{(\vec{y}_i - C \chi(A \vec{x}_i + \vec{b}))_{i \in I}} \\
\textrm{subject to} & (A,\vec{b},C) \in \rM_{D,N}(\Zp) \times \Zp^D \times \rM_{M,D}(\Qp)
\end{array}
\ee
Since the coefficients of $A$ and $\vec{b}$ are restricted to $\Zp$, the affine transformation preserves the compact subspace $\Zp^N \subset \Qp^N$, and hence we can apply $p$-adic optimisation methods based on topological generators.

\vs
As preceding studies, we recall two alternatives of the $p$-adic neural network. G.\ L.\ R.\ N'guessan invented an alternative frame work employing all functions in van der Put basis as activation functions in \cite{Ngu25}. A.\ P.\ Zubarev invented an alternative frame work replacing the varying affine transformations $A \vec{x} + \vec{b}$ by a single explicit transformation $\Zp^N \to \Zp$ for the case $N > 1$ but instead employing all continuous functions $\phi \colon \Zp \to \Qp$ as activation functions. In particular, both of the two alternatives employ infinitely many activation functions.

\vs
On the other hand, $p$-adic character neural network uses a single injective $p$-adic character $\chi$ as an activation functions. This benefit and the stability of the domain $\Zp^N$ by the affine transformations $A \vec{x} + \vec{b}$ enable us to reduce the $p$-adic character neural network to the feasibility problem of polynomial equations over the finite ring of integers modulo a power of $p$.

\vs
We briefly explain contents of this paper. In \S \ref{Convention}, we explain convention of time and space complexity in this paper. In \S \ref{p-adic Character}, we recall basic properties of $p$-adic characters. In \S \ref{Universal Approximation}, we show the $p$-adic universal approximation theorem for $p$-adic character neural network. In \S \ref{Formulation}, we formulate the $p$-adic character neural network, and reduce it to the feasibility problem of polynomial equations over the finite ring of integers modulo a power of $p$.

\section{Convention}
\label{Convention}

Throughout the paper, let $p$ be a prime number. We denote by $\Fp$ the finite field of integers modulo $p$, $\Zp$ the ring of $p$-adic integers, $\Qp$ the field of $p$-adic numbers, and $\v{\cdot} \colon \Qp \to [0,\infty)$ a $p$-adic absolute value. We note that the $p$-adic absolute value with the normalisation $\v{p} = p^{-1}$ is frequently denoted by $\v{\cdot}_p$, but we do not fix a normalisation.

\vs
We denote by $\N$ the set of non-negative integers. For a $d \in \Z$, we set $\N_{< d} \coloneqq \N \cap [0,d)$, and $\N_{\leq d} \coloneqq \N \cap [0,d]$. For sets $X$ and $Y$, we denote by $X^Y$ the set of maps $Y \to X$. We note that every $d \in \N$ is identified with $\N_{< d}$ in set theory, and hence $X^d$ formally means $X^{\N_{< d}}$, which is naturally identified with the set of $d$-tuples in $X$.

\vs
We denote by $t(r)$ (resp.\ $s(r)$) the worst-case time (resp.\ space) complexity of arithmetic operations $+,-,\times$ for elements in $\N_{< r}$ for an $r \in \N$, and by $t_f(e)$ (resp.\ $s_f(e)$) the worst-case time (resp.\ space) complexity of computation of $f$ modulo $p^e$ for an $e \in \N$. When we describe time complexity and space complexity of an algorithm, we ignore the computational complexity of increment of an index of a for-loop, as it just appears as $+t(n)$ and $+s(n)$ for the size $n$ of the loop.

\vs
For an $e \in \N$, the inverse of an $x \in (\Z/p^e \Z)^{\times}$ is computed in various ways, e.g.\ by Euler's theorem or Euclidean algorithm with time complexity $O((e \log_2 p)t(p^e))$ and space complexity $O(s(p^e))$, and Taylor expansion with time complexity $O((\log_2 p)t(p))$ and space complexity $O(p s(p))$ for preprocess and with time complexity $O(et(p^e))$ and space complexity $O(s(p^e))$ par query. Fixing an implementation of the computation of the inverse, we denote by $t_i(e)$ (resp.\ $s_i(e)$) its worst-case time (resp.\ space) complexity for an $e \in \N$.

\vs
When we write pseudocode, a for-loop along a subset of $\N$ denotes the loop of the ascending order, and a for-loop along $I$ denotes a loop in an arbitrary order.

\vs
For an $n \in \N$, a subset $S$ of $\N^n$, and a function $g \colon S \to \R$, we denote by $O(g)$ the set of functions $h \colon S \to \R$ satisfying that there is a $(C,R) \in (0,\infty)^2$ such that for any $x \in \N^n$ in the domains of both of $g$ and $h$, if every entry of $x$ is greater than $R$, then $\v{h(x)}$ is smaller than $Cg(x)$.

\section{$p$-adic Character}
\label{p-adic Character}

\begin{dfn}
Let $G$ be a topological group. A {\it $p$-adic character} on $G$ is a continuous group homomorphism $G \to \Qp^{\times}$. We denote by $G^{\vee}$ the set of $p$-adic characters on $G$, which forms an Abelian group with respect to the pointwise multiplication, and call it {\it the $p$-adic Pontryagin dual} of $G$.
\end{dfn}

We recall the structure of the $p$-adic Pontryagin dual $\Zp^{\vee}$ of the additive group $\Zp$. For any $a \in 1 + p \Zp$, the group homomorphism
\be
\Z & \to & \Qp^{\times} \\
i & \mapsto & a^i
\ee
uniquely extends to a $p$-adic character
\be
a^{\bullet} \colon \Zp & \to & \Qp^{\times} \\
x & \mapsto & a^x \coloneqq \sum_{e=0}^{\infty} \binom{x}{e} (a-1)^e
\ee
by binomial coefficient theorem, where $\binom{x}{d}$ denotes the binomial coefficient function $(d!)^{-1} \prod_{i=0}^{d-1} (x-i)$ of degree $d$.

\vs
On the other hand, for any $p$-adic character $\chi$ on $\Zp$, its restriction to $\Z$ is characterised by the image $\chi(1)$ of $1$, which belongs to $1 + p \Zp$ by the continuity of $\chi$, and hence we have $\chi = \chi(1)^{\bullet}$. Therefore, the map
\be
1 + p \Zp & \to & \Zp^{\vee} \\
a & \mapsto & a^{\bullet}
\ee
is bijective. In fact, it is a group isomorphism with respect to the multiplication on $1 + p \Zp \subset \Qp^{\times}$.

\begin{exm}
We introduce three important examples of $p$-adic characters on $\Zp$.
\bi
\item[(1)] {\it The trivial $p$-adic character} on a topological group $G$ is the constant map $G \to \Qp^{\times}$ with value $1$. The trivial $p$-adic character on $\Zp$ is presented as $1^{\bullet}$.
\item[(2)] When $p = 2$, then we have $-1 \in 1 + 2 \Z_2$, and hence the $p$-adic character $(-1)^{\bullet}$ on $\Zp$ makes sense. By the definition, we have $((-1)^{\bullet})^2 = 1^{\bullet}$, and hence $(-1)^{\bullet}$ is of order $2$.
\item[(3)] The convergent power series
\be
\exp_p(x) \coloneqq \sum_{e=0}^{\infty} \frac{1}{e!} x^e \in \Q[[x]]
\ee
converges in
\be
\set{x \in \Qp}{\v{x} < \v{p}^{1/(p-1)}} = q \Zp,
\ee
where $q = 4$ when $p = 2$ and $q = p$ otherwise, and defines a $p$-adic character $\exp_p$ on the additive group $q \Zp$. We call $\exp_p$ {\i the $p$-adic exponential}. In particular, $\exp_p(qx)$ is a $p$-adic character on $\Zp$ with the Taylor expansion
\be
\exp_p(qx) = \sum_{e=0}^{\infty} \frac{q^e}{e!} x^e,
\ee
and coincides with $\exp_p(q)^{\bullet}$.
\ei
\end{exm}

There are three ways to compute $p$-adic characters modulo $p^E$ for a fixed precision parameter $E \in \N \setminus \ens{0}$. One is the direct application of the definition using the binomial expansion, i.e.\ the Mahler series (cf.\ \cite{Mah58} Theorem 1). Another one is the application of Taylor expansion. The other one is the binary modular method.

\vs
For all three methods, we may replace the arguments $a \in 1 + p \Zp$ and $x \in \Zp$ by $a \bmod p^E \in (1 + p \Zp) \cap \N_{< p^E}$ and $x \bmod p^{E-1} \in \N_{< p^{E-1}}$, because we have $a^x \in 1 + p^E \Zp$ when $a \in 1 + p^E \Zp$ or $x \in p^{E-1} \Zp$.

\vs
For the first two methods, it is sometimes convenient to preprocess the computation of factorials and their inverses modulo $p^E$ up to $N!$ for a fixed $N \in \N$. For this purpose, it suffices to compute the $p$-adic valuations of factorials and $p$-coprime parts of them modulo $p^E$. The preprocess can is executed with time complexity $O((p-1)^{-1}N s(\max \ens{p^E,N})$ and space complexity $O(N t(p^E))$ after the computation of $p^E$. Here are pseudocodes for the preprocess:

\begin{figure}[H]
\begin{algorithm}[H]
\caption{Computation of $p$-adic valuation and $p$-coprime part of $x \in \Zp$}
\label{coprime}
\begin{algorithmic}[1]
\Function {ValuativeDecomposition}{$p,x$}
	\State $v \gets 0$
	\While {$x \bmod p = 0$}
		\State $x \gets p^{-1}x$
		\State $v \gets v + 1$
	\EndWhile
	\State \Return $(v,x)$
\EndFunction
\end{algorithmic}
\end{algorithm}
\end{figure}

\begin{figure}[H]
\begin{algorithm}[H]
\caption{Preprocessing the computation of the $p$-adic valuations of factorials and $p$-coprime parts of them modulo $p^E$ up to $N!$}
\label{factorial}
\begin{algorithmic}[1]
\Function {ModularFactorial}{$p,E,N$}
	\State $\vec{f} = (f_e)_{e=0}^{N}\gets ((0,1))_{e=0}^{N}$
	\ForAll {$e \in \N_{< N}$}
		\State $(v,u) \gets f_e$
		\State $(v',u') \gets$ \Call{ValuativeDecomposition}{$p,e+1$}
		\State $f_{e+1} \gets (v+v',uu' \bmod p^E)$
	\EndFor
	\State \Return $\vec{f}$
\EndFunction
\end{algorithmic}
\end{algorithm}
\end{figure}

For the first method, the Mahler series $\sum_{e=0}^{\infty} \binom{x}{e} (a-1)^e$ can be truncated to be $\sum_{e=0}^{E-1} \binom{x}{e} (a-1)^e$ because $(a-1)^e \in p^e \Zp$ for any $e \in \N$. The process is executed with time complexity $O(((p-1)^{-1}x+E)t(p^E) + Et_i(E))$ and space complexity $O(s(p^E) + s_i(E))$ after the computation of $p^v$ for each $v \in \N_{\leq E}$. Here is a pseudocode for this process:

\begin{figure}[H]
\begin{algorithm}[H]
\caption{Computation of $a^x$ modulo $p^E$ for $a \in (1 + p \Zp) \cap \N_{< p^E})$ and $x \in \N_{< p^{E-1}}$ using the binomial expansion}
\label{character Mahler series}
\begin{algorithmic}[1]
\Function {CharacterMahlerSeries}{$p,E,a,x$}
	\State $y \gets 0$ \Comment{variable for the return value}
	\State $(v,u) \gets (0,1)$ \Comment{variables for the $p$-adic valuation and the $p$-coprime part of $\binom{x}{e}$}
	\State $b \gets 1$ \Comment{variable for $(a-1)^e$}
	\ForAll {$e \in \N_{< E}$}
		\If {$e > 0$}
			\State $(v',u') \gets$ \Call{ValuativeDecomposition}{$p,x-e+1$}
			\State $(v'',u'') \gets$ \Call{ValuativeDecomposition}{$p,e$}
			\State $u'' \gets$ the inverse of $u''$ modulo $p^{E-e}$
			\State $(v,u) \gets (v+v'-v'',uu'u'' \bmod p^{E-e})$
			\State $b \gets b(a-1) \bmod p^E$
		\EndIf
		\State $y \gets (y + p^vub) \bmod p^E$
	\EndFor
	\State \Return $y$
\EndFunction
\end{algorithmic}
\end{algorithm}
\end{figure}

We note that $p^vub$ in the process above is congruent to $0$ if $v \geq E$. Therefore, we need only to preprocess the computation of $p^v$ for all $v \in \N_{\leq E}$.

\vs
For the second method, we only consider the $p$-adic character $\exp_p(qx)$, because $a^{\bullet}$ can be expressed as
\be
\left\{
\begin{array}{ll}
\exp_p((\log_p a)x) & (p \neq 2 \lor a \in 1 + q \Z_p \lor x \in p \Zp) \\
- \exp_2((\log_2 -a)x) & (p = 2 \land a \in -1 + 4 \Z_2 \land x \in 1 + 2 \Z_2) \\
\end{array}
\right.
\ee
for any $a \in 1 + p \Zp$, where $\log_p$ denotes the map $1 + p \Zp \to \Qp$ called {\it the Iwasawa logarithm} defined by the convergent power series
\be
\log_p x \coloneqq \sum_{e=1}^{\infty} \frac{-(1-x)^e}{e} \in \Q[[x-1]]
\ee
and its values in the expressions above belongs to $q \Zp$.

\vs
Set  $m \coloneqq 2$ when $p = 2$ and $m \coloneqq 1$ otherwise, and $E' \coloneqq \lceil (m - (p-1)^{-1})^{-1}E \rceil$. The Taylor series $\sum_{e=0}^{\infty} \frac{q^e}{e!} x^e$ can be truncated to be $\sum_{e=0}^{E'-1} \frac{q^e}{e!} x^e$ because $\frac{q^e}{e!} \in q^ep^{- \lfloor (p-1)^{-1}e \rfloor} \Zp = p^{\lceil (m-(p-1)^{-1})e \rceil} \Zp$ for any $e \in \N$ by Legendre's formula. By the same reasoning, we may replace the argument $x \in \Zp$ by $x \bmod p^{E-m} \in \N_{< p^{E-m}}$. The process is executed with time complexity $O(E(t(p^E)+t_i(E)))$ and space complexity $O(s(p^E)+s_i(E))$ after the computation of $p^v$ for each $v \in \N_{\leq E}$. Here is a pseudocode for this process:

\begin{figure}[H]
\begin{algorithm}[H]
\caption{Computation of $\exp_p(qx)$ modulo $p^E$ for $a \in (1 + p \Zp) \cap \N_{< p^E})$ and $x \in \N_{< p^{E-m}}$ using the Taylor expansion}
\label{character Taylor series}
\begin{algorithmic}[1]
\Function {CharacterTaylorSeries}{$p,E,a,x$}
	\State $y \gets 0$ \Comment{variable for the return value}
	\State $(v,u) \gets (0,1)$ \Comment{variables for the $p$-adic valuation and the $p$-coprime part of $(e!)^{-1}$}
	\If {$p = 2$}
		\State $m \gets 2$
	\Else
		\State $m \gets 1$
	\EndIf
	\State $q \gets p^m$
	\State $E' \gets \lceil (m - (p-1)^{-1})^{-1}E \rceil$
	\State $b \gets 1$ \Comment{variable for $x^e$}
	\ForAll {$e \in \N_{< E'}$}
		\If {$e > 0$}
        \State $e' = \max \ens{0,E-em}$
			\State $(v',u') \gets$ \Call{ValuativeDecomposition}{$p,e$}
			\State $u' \gets$ the inverse of $u'$ modulo $p^{e'}$
			\State $(v,u) \gets (v-v',uu' \bmod p^{e'})$
			\State $b \gets bx \bmod p^{e'}$
		\EndIf
		\State $y \gets (y + p^{v+me}ub) \bmod p^E$
	\EndFor
	\State \Return $y$
\EndFunction
\end{algorithmic}
\end{algorithm}
\end{figure}

For the third method, since we have already reduced it to the case $a \in \N_{< p^E}$ and $x \in \N_{p^{E-1}}$, we have nothing to improve from the classical binary modular method. The process is executed with time complexity $O((\log_2 x) t(p^E))$ and space complexity $O(s(p^E))$ after the computation of $p^E$. Here is a pseudocode for this process:

\begin{figure}[H]
\begin{algorithm}[H]
\caption{Computation of $a^x$ modulo $p^E$ for $a \in (1 + p \Zp) \cap \N_{< p^E})$ and $x \in \N_{< p^{E-1}}$ using binary modular method}
\label{character binomial modular method}
\begin{algorithmic}[1]
\Function {CharacterBinaryModular}{$p,E,a,x$}
	\State $y \gets 0$ \Comment{variable for the return value}
	\While {$x \neq 0$}
		\If {$x \bmod 2 = 1$}
			\State $y \gets yx \bmod p^E$
		\EndIf
		\State $x \gets x^2 \bmod p^E$
		\State $x \gets \lfloor \frac{x}{2} \rfloor$
	\EndWhile
	\State \Return $y$
\EndFunction
\end{algorithmic}
\end{algorithm}
\end{figure}

We note that the expressions $x \bmod 2$ and $\lfloor \frac{x}{2} \rfloor$ in the process above naturally make sense even when $p \neq 2$, because $x$ is reduced to a natural number rather than a general $p$-adic integer.

\vs
If we only care about the simplicity, it suffices to use the third method. However, the first method has the benefit that it gives information of the Mahler series, which is useful when we need to deduce a $p$-adic continuous function from sample data possibly with noise small with respect to the supremum norm. The second method has the benefit that it gives information of coefficients of polynomial approximation, which are useful when we apply methods for polynomials such as $p$-adic Newton's method. It is good to choose the most suitable method for the purpose of the use of a $p$-adic character.

\vs
We recall a characterisation of the injectivity of a $p$-adic character.

\begin{prp}
\label{injectivity}
Let $\chi$ be a $p$-adic character on $\Zp$. The the following are equivalent:
\bi
\item[(1)] The character $\chi$ is nether $1^{\bullet}$ nor $(-1)^{\bullet}$ for the case $p = 2$.
\item[(2)] The character $\chi$ is injective.
\ei
\end{prp}

\begin{proof}
The implication from (2) to (1) is obvious. We show that the negation of (2) implies the negation of (1). By the assumption that $\chi$ is not injective, there exists some $x \in \ker(\chi) \setminus \ens{0}$. We have $\Z x \subset \ker(\chi)$, and hence $\Zp x \subset \ker(\chi)$ by the closedness of $\ker(\chi)$. By $x \neq 0$, $\Zp x$ is a non-zero ideal of $\Zp$, and hence is of the form $p^e \Zp$ for some $e \in \N$. In particular, $\chi$ induces an injective group homomorphism $\Z/p^e \Z \hookrightarrow \Qp^{\times}$. The residue class $1 + p^e \Z \in \Z/p^e \Z$ of $1$ is of order $p^e$, and hence so is its image in $\Qp^{\times}$. However, the multiplicative group
\be
\Qp^{\times} \cong 
\left\{
\begin{array}{ll}
\Z \times \ens{1,-1} \times (1 + 4 \Z_2) & (p = 2) \\
\Z \times \Fp^{\times} \times (1 + p \Zp) & (p \neq 2)
\end{array}
\right.
\ee
admits at most two elements of $p$-power order: $1$ and $-1$ for the case $p = 2$. The former case corresponds to the case $e = 0$ and $\chi = 1^{\bullet}$, and the latter case corresponds to the case $(p,e) = (2,1)$ and $\chi = (-1)^{\bullet}$.
\end{proof}

When a given $p$-adic character $\chi$ is injective, it is expressed as $a^{\bullet}$ for some $a \in (1 + p \Zp) \setminus \ens{1,-1}$ by Proposition \ref{injectivity}. The inverse of the $p$-adic character whose codomain is restricted to its image $a^{\Zp} \subset 1 + p \Zp$ is given by the group homomorphism
\be
a^{\Zp} & \to & \Zp \\
y & \mapsto & 
\left\{
\begin{array}{ll}
\frac{\log_p y}{\log_p a} & (y \in 1 + q \Zp ) \\
\frac{\log_p -y}{\log_p -a} & (y \in -1 + q \Zp)
\end{array}
\right..
\ee
We note that the latter case occurs only when $p = 2$ and $a \in -1 + q \Zp$. Using the inverse map, the deduction of arguments of a $p$-adic character is reduced to the deduction of its values.

\section{Universal Approximation}
\label{Universal Approximation}

For any $\Qp$-valued function $f$ on $\Zp$ and any $\vec{x} = (x_i)_{i=0}^{n-1} \in\Zp^n$ with $n \in \N$, we abbreviate $(f(x_i))_{i=0}^{n-1} \in \Qp^n$ to $f(\vec{x})$. Let $\chi$ be an injective $p$-adic character on $\Zp$.

\begin{thm}[Universal approximation theorem for a $p$-adic character]
\label{universal approximation theorem}
For any $(n,m) \in \N^2$, any continuous map $f \colon \Zp^n \to \Qp^m$, and any $\epsilon \in \R_{>0}$, there exists a tuple $(d,A,\vec{b},C)$ of a $d \in \N$, an $A \in \rM_{d,n}(\Zp)$, a $\vec{b} \in \Zp^d$, and a $C \in \rM_{m,d}(\Qp)$ such that for any $\vec{x} \in \Zp^n$, the inequality $\n{f(\vec{x}) - C \chi(A \vec{x} + \vec{b})} < \epsilon$ holds with respect to the $\ell^{\infty}$-norm $\n{\cdot}$ on $\Qp^m$.
\end{thm}

\begin{proof}
We note that the assertion is true even if we additionally assume that $\vec{b}$ should be chosen to be the zero vector. We first show the assertion for the case $m = 1$. We denote by $R \subset \rC(\Zp^n,\Qp)$ the subset of functions $g$ such that there exists a tuple $(d,A,C)$ of a $d \in \N$, an $A \in \rM_{d,n}(\Zp)$, and a $C \in \rM_{1,d}(\Qp)$ such that for any $\vec{x} \in \Zp^n$, the equality $g(\vec{x}) = C \chi(A \vec{x})$ holds.

\vs
By $p$-adic Stone--Weierstrass theorem (cf.\ \cite{Kap50} Theorem and \cite{Ber90} 9.2.5.\ Theorem), it suffices to show that $R$ is a $\Qp$-subalgebra of $\rC(\Zp^n,\Qp)$ separating points of $\Zp^n$. We show that $R$ separates points of $\Zp$. For any $\vec{x}_0, \vec{x}_1 \in \Zp^n$ with $\vec{x}_0 \neq \vec{x}_1$, there exists a $g \in R$ such that $g(\vec{x}_0) \neq g(\vec{x}_1)$. Indeed, by $\vec{x}_0 \neq \vec{x}_1$, there exists an $i \in \N_{< n}$ such that the $(1+i)$-th entry $x_{0,i}$ of $\vec{x}_0$ is distinct from the $(1+i)$-th entry $x_{1,i}$ of $\vec{x}_1$, and we have
\be
\chi(P_i \vec{x}_0) = \chi(x_{0,i}) \neq \chi(x_{1,i}) = \chi(P_i \vec{x}_1)
\ee
by the injectivity of $\chi$, where $P_i \in \rM_{1,n}(\Zp)$ denotes the $(1+i)$-th projection.

\vs
We show that $R$ is closed under addition and multiplication. Let $(g_0,g_1) \in R^2$. By the definition of $R$, there exist a tuple $(d_s,A_s,C_s)$ of a $d_s \in \N$, an $A_s = ((A_{s,i,j})_{j=0}^{n-1})_{i=0}^{d_s-1} \in \rM_{d_s,n}(\Zp)$, and a $C_s = ((C_{s,i,j})_{j=0}^{d_s-1})_{i=0}^{0} \in \rM_{1,d_s}(\Qp)$ such that for any $\vec{x} \in \Zp^n$, the equality $g_s(\vec{x}) = C_s \chi(A_s \vec{x})$ holds for each $s \in \{0,1\}$. We have
\be
(g_0 + g_1)(\vec{x}) & = & C_0 \chi(A_0 \vec{x}) + C_1 \chi(A_1 \vec{x}) \\
& = & \sum_{s=0}^{1} \sum_{i=0}^{0} \sum_{j=0}^{d_s-1} C_{s,i,j} \chi \left( \sum_{k=0}^{n-1} A_{s,j,k} x_k \right) \\
& = & \sum_{s=0}^{1} \sum_{j=0}^{d_s-1} C_{s,0,j} \chi \left( \sum_{k=0}^{n-1} A_{s,j,k} x_k \right) \\
& = & \left( 
\begin{array}{cc}
C_0 & C_1
\end{array}
\right) \chi \left( \left( 
\begin{array}{c}
A_0 \\
A_1
\end{array}
\right) \vec{x} \right)
\ee
and
\be
(g_0 g_1)(\vec{x}) & = & C_0 \chi(A_0 \vec{x}) C_1 \chi(A_1 \vec{x}) \\
& = & \prod_{s=0}^{1} \sum_{i=0}^{0} \sum_{j=0}^{d_s-1} C_{s,i,j} \chi \left( \sum_{k=0}^{d_s-1} A_{s,j,k} x_k \right) \\
& = & \sum_{j_0=0}^{d_0-1} \sum_{j_1=0}^{d_1-1} C_{0,0,j_0} C_{1,0,j_1} \chi \left( \sum_{k=0}^{n-1} A_{0,j,k} x_k \right) \chi \left( \sum_{k=0}^{n-1} A_{1,j,k} x_k \right) \\
& = & \sum_{j_0=0}^{d_0-1} \sum_{j_1=0}^{d_1-1} C_{0,0,j_0} C_{1,0,j_1} \chi \left( \sum_{k=0}^{n-1} A_{0,j_0,k} x_k + \sum_{k=0}^{n-1} A_{1,j_1,k} x_k \right) \\
& = & \sum_{j_0=0}^{d_0-1} \sum_{j_1=0}^{d_1-1} C_{0,0,j_0} C_{1,0,j_1} \chi \left( \sum_{k=0}^{n-1} (A_{0,j_0,k} + A_{1,j_1,k}) x_k \right) \\
& = & ((C_{0,0,j \bmod d_0} C_{1,0,\lfloor j/d_0 \rfloor})_{j=0}^{d_0 d_1 - 1})_{i=0}^{0} \chi \left( ((A_{0,j \bmod d_0,k} + A_{1,\lfloor j/d_0 \rfloor,k})_{k=0}^{n-1})_{j=0}^{d_0 d_1 - 1} \vec{x} \right)
\ee
for any $\vec{x} \in \Zp^n$, and hence $g_0 + g_1, g_0 g_1 \in R$. By the connectivity of scalar multiplication and matrix multiplication, $R$ is closed under scalar multiplication by $\Qp$. Therefore, $R$ is a $\Qp$-subalgebra of $\rC(\Zp^n,\Qp)$.

\vs
We next consider the general case where $m$ is not necessarily $1$. By the argument above for the case $m = 1$, there exists a tuple $(\vec{d},\vec{A},\vec{C})$ of a $\vec{d} = (d_j)_{j=0}^{m-1} \in \N^m$, a $\vec{A} = (A_j)_{j=0}^{m-1} \in \prod_{j=0}^{m-1} \rM_{d_j,n}(\Zp)$, and a $\vec{C} = (C_j)_{j=0}^{m-1} \in \prod_{j=0}^{m-1} \rM_{1,d_j}(\Qp)$ such that for any $\vec{x} \in \Zp^n$, the inequality $\n{f(\vec{x}) - (C_j \chi(A_j \vec{x}))_{j=0}^{m-1}} < \epsilon$ holds.

\vs
Setting
\be
d & \coloneqq & \sum_{j=0}^{m-1} d_j \\
A & \coloneqq & \left(
\begin{array}{c}
A_0 \\
A_1 \\
\vdots \\
A_{m-1}
\end{array}
\right) \in \rM_{d,n}(\Zp) \\
C & \coloneqq & \left(
\begin{array}{ccccc}
C_0 & O & \cdots & O \\
O & C_1 & \cdots & O \\
\vdots & \ddots & \ddots & \vdots \\
O & O & \cdots & C_{m-1}
\end{array}
\right) \in \rM_{m,d}(\Qp),
\ee
we have
\be
(C_j \chi(A_j \vec{x}))_{j=0}^{m-1} = \left(
\begin{array}{c}
C_0 \chi(A_0 \vec{x}) \\
C_1 \chi(A_1 \vec{x}) \\
\vdots \\
C_{m-1} \chi(A_{m-1} \vec{x})
\end{array}
\right) = C \chi(A \vec{x}),
\ee
and hence the inequality $\n{f(\vec{x}) - C \chi(A \vec{x})} < \epsilon$ holds for any $\vec{x} \in \Zp^n$.
\end{proof}

\section{Formulation}
\label{Formulation}

By Theorem \ref{universal approximation theorem}, we obtain a new formulation of a $p$-adic neural network. Let $\chi$ be an injective $p$-adic character on $\Zp$, $I$ a finite set, $X = (\vec{x}_i)_{i \in I} \in (\Zp^N)^I$ a sequence of sample points with $N \in \N$, $Y = (\vec{y}_i)_{i \in I} \in (\Qp^M)^I$ a sequence of observed values corresponding to $X$ with $M \in \N$, and $D \in \N_{> 0}$ a hyperparameter for dimension. {\it The $p$-adic character neural network} is the following optimisation problem:
\be
\begin{array}{ll}
\textrm{minimise} & \n{(\vec{y}_i - C \chi(A \vec{x}_i + \vec{b}))_{i \in I}} \\
\textrm{subject to} & (A,\vec{b},C) \in \rM_{D,N}(\Zp) \times \Zp^D \times \rM_{M,D}(\Qp),
\end{array}
\ee
where the norm in the target expression of the minimisation is an arbitrary fixed norm on $\Qp^I$ such as the $\ell^{\infty}$-norm or the $\ell^1$-norm.

\vs
Practically, we additionally need to set hyperparameters $E \in \N_{> 0}$ and $F \in \N$ for precision and an upperbound of the $p$-adic valuations of denominators, and assume the following:
\bi
\item[(1)] Each entry of $X$ is observed up to congruence modulo $p^{E+F-1}$.
\item[(2)] Each entry of $Y$ is observed up to congruence modulo $p^E$.
\item[(3)] Each entry of $Y$ belongs to $p^{-F} \Zp \subset \Qp$.
\ei
Then the $p$-adic neural network is reformulated as the following approximated optimisation problem:
\be
\begin{array}{ll}
\textrm{minimise} & \n{(\vec{y}_i - C \chi(A \vec{x}_i + \vec{b}))_{i \in I} \bmod p^E} \\
\textrm{subject to} & (A,\vec{b},C) \in \rM_{D,N}(\N_{< p^{E+F-1}}) \times \N_{< p^{E+F-1}}^D \times p^{-F} \rM_{M,D}(\N_{< p^{E+F}}),
\end{array}
\ee
where $Z \bmod p^E$ for a $Z \in p^{-F} (\Zp^M)^I$ denotes a unique element of $p^{-F} (\N_{< p^{E+F}}^M)^I$ congruent to $Z$ modulo $p^E (\Zp^M)^I$.

\vs
As we have explained in \S \ref{Introduction}, the benefit of $p$-adic character neural network against the original $p$-adic neural network based on characteristic functions of clopen subsets is that we avoid to use denominators for $A$ and $\vec{b}$, which cause scaling of the arguments and prevents uniform approximation by polynomials.

\vs
Thanks to this benefit, the second method of the computation of $\chi$ explained in \S \ref{p-adic Character} allows us to express $(\vec{y}_i - C \chi(A \vec{x}_i + \vec{b}))_{i \in I} \bmod p^E$ as a sequence of polynomials on entries of $A$, $\vec{b}$, and $p^F C$. Therefore, multiplying $p^F$ and replacing $E$ by $E + F$, the $p$-adic character neural network is reduced to the following optimisation problem for a given sequence $(f_i)_{i \in I}$ of $L$-variable polynomials over $\Zp$ with $L \in \N$:
\be
\begin{array}{ll}
\textrm{minimise} & \n{(f_i(\vec{z}))_{i \in I} \bmod p^E} \\
\textrm{subject to} & \vec{z} \in \N_{< p^E}^L
\end{array}
\ee
When we consider the $\ell^1$-norm on $\Qp^I$, the minimisation problem for the simple case where $f_i$ is of degree $\leq 1$ for any $i \in I$ and $E = 1$ is the maximal feasible subsystem problem for linear equations over $\Fp$, which is APX-complete, i.e.\ complete for the class of problems which allow constant-factor approximations, by \cite{AK95} Proposition A.1. As a future study, we consider effective heuristic algorithms of $p$-adic character neural network.

\vs
When we consider $\ell^{\infty}$-norm on $\Qp^I$, the minimisation problem is reduced to the feasibility problem of polynomial equations over $\Z/p^e \Z$ for $e \in \N_{\leq E}$. Indeed, if there is an $e \in \N_{\leq E}$ such that the system
\be
\forall i \in I, f_i(\vec{z})\equiv 0 \pmod{p^e}
\ee
has not solution $\vec{z} \in \N_{< p^e}$, then the minimum of the target expression $\n{(f_i(\vec{z}))_{i \in I} \bmod p^E}$ is expressed as $\v{p}^{e_{\max} - 1}$, where $e_{\max}$ denotes the maximum of such an $e$. If there is no such $e$, then the system for $e = E$ has a solution, and hence the minimum coincides with $0$.

\vs
In order to solve the feasibility problem of polynomial equations over the principal ideal ring $\Z/p^e \Z$ for a single $e \in \N$, it suffices to compute a Gr\"obner basis (cf.\ \cite{NS01} for the formulation of a Gr\"obner basis of an ideal of polynomials over a principal ideal ring and \cite{KK25} \S 3 for the application of a Gr\"obner basis to the feasibility problem of polynomial equations over a principal ideal ring).

\vs
In order to solve the feasibility problem of polynomial equations over $\Z/p^e \Z$ for varying $e \in \N$, it suffices to apply digit dynamic programming with base $p$. Indeed, every common zero in $\Z/p^{e+1} \Z$ gives a common zero in $\Z/p^e \Z$ for any $e \in \N$, and hence the set of all common zero can be computed inductively on $e$. Here is a pseudocode for the process:

\begin{figure}[H]
\begin{algorithm}[H]
\caption{Computation of the maximum of an $e \in \N$ for which a given system $\vec{f} = (f_i(\vec{x}))_{i \in I}$ of polynomials with $L$-variables over $\Z$ has a common zero in $\Z/p^e \Z$ by digit dynamic programming with base $p$}
\begin{algorithmic}[1]
\Function {DigitDynamicProgramming}{$p,L,\vec{f}$}
	\State $e \gets 0$
	\State $Z \gets$ the array with the single entry $(0)_{l=0}^{L-1}$ \Comment{variable for an array of common zeros of $\vec{f}$ modulo $p^e$}
	\While {True}
		\State $W \gets$ the empty array \Comment{variable for an array of common zeros of $\vec{f}$ modulo $p^{e+1}$}
		\ForAll {entry $\vec{z}$ of $Z$}
			\ForAll {$\vec{d} \in \N_{<p}^L$}
				\State $\vec{w} \gets \vec{z} + p^e \vec{d}$
				\If {$\vec{w}$ is a common zero of $\vec{f}$ modulo $p^{e+1}$}
					\State Append $\vec{w}$ to $W$
				\EndIf
			\EndFor
		\EndFor
		\If {$W$ is empty}
			\State \Break
		\EndIf
		\State $e \gets e + 1$
		\State $Z \gets W$
	\EndWhile
	\State \Return $e$
\EndFunction
\end{algorithmic}
\end{algorithm}
\end{figure}

We note that this process results in an infinite loop if $\vec{f}$ has a common zero in $\Z/p^e \Z$ for all $e \in \N$, or equivalently, has a common zero in $\Zp$. If we need to know whether $\vec{f}$ has a common zero in $\Z/p^E \Z$ or not and the maximum of an $e \in \N_{< E}$ for which $\vec{f}$ has a common zero in $\Z/p^e \Z$, then it suffices to break the while loop when $e = E$.

\vspace{0.3in}
\addcontentsline{toc}{section}{Acknowledgements}
\noindent {\Large \bf Acknowledgements}
\vspace{0.2in}

\noindent
I thank all people who helped me to learn mathematics and programming. I also thank my family.

%

\end{document}